\documentclass[12pt,reqno]{amsart}

\numberwithin{equation}{section} 

\usepackage{enumerate}
\usepackage{amssymb}
\usepackage{amsmath}
\usepackage{amscd}
\usepackage{amsthm}
\usepackage{amsfonts}
\usepackage{graphicx}
\usepackage[all]{xy}
\usepackage{verbatim}
\usepackage{hyperref}

\newtheorem{theorem}{Theorem}[section]

\newtheorem{lemma}[theorem]{Lemma}
\newtheorem{corollary}[theorem]{Corollary}
\newtheorem{definition}[theorem]{Definition}

\newtheorem{conjecture}[theorem]{Conjecture}

\newtheorem{proposition}[theorem]{Proposition}
\newtheorem{notation}[theorem]{Notation}
\newtheorem{problem}[theorem]{Problem}

\newcommand{\al}{\alpha}
\newcommand{\be}{\beta}
\newcommand{\ga}{\gamma}
\newcommand{\Ga}{\Gamma}
\newcommand{\del}{\delta}

\newcommand{\lam}{\lambda}
\newcommand{\Lam}{\Lambda}
\newcommand{\eps}{\epsilon}

\newcommand{\sig}{\sigma}

\newcommand{\om}{\omega}
\newcommand{\Om}{\Omega}
\newcommand{\vphi}{\varphi}

\newcommand{\cE}{\mathcal{E}}

\newcommand{\cN}{\mathcal{N}}

\newcommand{\cS}{\mathcal{S}}

\newcommand{\bR}{\mathbb{R}}
\newcommand{\bZ}{\mathbb{Z}}
\newcommand{\bQ}{\mathbb{Q}}

\newcommand{\bN}{\mathbb{N}}

\newcommand{\gog}{\mathfrak{g}}
\newcommand{\goh}{\mathfrak{h}}

\newcommand{\gou}{\mathfrak{u}}

\newcommand\av[1]{\left|#1\right|}
\newcommand\set[1]{\left\{#1\right\}}
\newcommand\pa[1]{\left(#1\right)}

\newcommand\idist[1]{\langle#1\rangle}

\newcommand{\onto}{\xymatrix{\ar@{>>}[r]&}}
\newcommand{\da}[4]{\xymatrix{#1 \ar@<.5ex>[r]^{#2} \ar@<-.5ex>[r]_{#3} & #4}}

\begin{document}
\title{On a generalization of Littlewood's conjecture}
\author{Uri Shapira}
\thanks{* Part of the author's Ph.D thesis at the Hebrew University of Jerusalem.\\ Email: ushapira@gmail.com}
\begin{abstract}
We present a class of lattices in $\bR^d$ ($d\ge 2$) which we call $GL-lattices$ and conjecture that any lattice is such. This conjecture is referred to as GLC.  Littlewood's conjecture amounts to saying that $\bZ^2$ is GL. We then prove existence of GL lattices by first establishing a dimension bound for the set of possible exceptions. Existence of vectors ($GL-vectors$) in $\bR^d$ with special Diophantine properties is proved by similar methods. For dimension $d\ge 3$ we give explicit constructions of GL lattices (and in fact a much stronger property). We also show that GLC is implied by a conjecture of G. A. Margulis concerning bounded orbits of the diagonal group. The unifying theme of the methods is to exploit rigidity results in dynamics (\cite{EKL},\cite{B},\cite{LW}), and derive results in Diophantine approximations or the geometry of numbers. 
\end{abstract}

\maketitle
\tableofcontents
%
%
%
%
\section{Notation Results and conjectures}
\label{sec-results}
We first fix our notation and define the basic objects to be discussed in this paper. Throughout this paper $d\ge 2$ is an integer.
Let $X_d$ denote the space of $d$-dimensional unimodular lattices in $\bR^d$ and let $Y_d$ denote the space of translates of such lattices. Points of $Y_d$ will be referred to as \textit{grids}, hence for $x\in X_d$ and $v\in\bR^d$, $y=x+v\in Y_d$ is the grid obtained by translating the lattice $x$ by the vector $v$. We denote by $\pi$  the natural projection 
\begin{equation}\label{projections}
 Y_d \stackrel{\pi}{\longrightarrow} X_d,\quad x+v  \mapsto x.
 \end{equation}
 In the next section we shall see that $X_d,Y_d$ are homogeneous spaces and equip them with metrics.
 Let $N:\bR^d\to \bR$ denote the function $N(w)=\prod_1^dw_i$ (note that we do not recall the dimension $d$ in the notation). 
For each $x\in X_d$, we identify the fiber $\pi^{-1}(x)$ in $Y_d$ with the torus $\bR_d/x$. This enables us to define an operation of multiplication by an integer $n$ on $Y_d$, i.e. if $y=x+v\in Y_d$ then $ny=x+nv$. For a grid $y\in Y_d$, we denote
\begin{equation}\label{mu}
N(y)= \inf\set{\av{N(w)} : w\in y }.
\end{equation}
The function $N:Y_d\to\bR$ will be of most interest to us.
The main objective of this paper is to discuss the following generalization of Littlewood's conjecture, referred to in this paper as GLC:
\begin{conjecture}[GLC]\label{con:Bigwood conjecture}
For any $d\ge 2$ and $y\in Y_d$ 
\begin{equation}\label{eq:Bigwood conjecture}
\inf_{n\ne 0}\av{nN\pa{ny}}=0.
\end{equation}
\end{conjecture}
\begin{definition}
\begin{enumerate}
\item A grid $y\in Y_d$ is L (\textit{Littlewood}) if \eqref{eq:Bigwood conjecture} holds. 
\item A lattice  $x\in X_d$ is GL if any grid $y\in\pi^{-1}(x)$ is L.
\item A vector $v\in \bR^d$ is GL if for any $x\in X_d$, the grid $y=x+v$ is L. 
\end{enumerate}
\end{definition}
Thus conjecture \ref{con:Bigwood conjecture} could be rephrased as saying that any lattice (resp vector) is GL. Of particular interest are GL lattices and vectors. 
For example, when $d=2$, $x=\bZ^2\in X_2$ and $v=(\al,\be)^t\in\bR^2$, the reader should untie the definitions and see that the grid $x+v$ satisfies ~\eqref{eq:Bigwood conjecture}, if and only if $\inf_{n\ne 0} n\idist{n\al}\idist{n\be}=0$ (where we denote for $\ga\in\bR$,  $\idist{\ga}=inf_{m\in \bZ}\av{\ga-m}$), i.e. if and only if the numbers $\al,\be$ satisfy the well known Littlewood conjecture (see \cite{Ma1}). Thus this conjecture could be stated as follows:
\begin{conjecture}[Littlewood]
The lattice $\bZ^2$ is GL.
\end{conjecture}
In this paper we shall prove existence of both GL lattices and vectors in any dimension $d\ge 2$. This result will follow from a dimension  bound on the set of exceptions to GLC. \\
 
$SL_d(\bR)$ and its subgroups acts naturally (via its linear action on $\bR^d$) on the spaces in \eqref{projections} in an equivariant fashion. Of particular interest to us will be the action of the group $A_d$ of $d\times d$ diagonal matrices with positive diagonal entries and determinant one. This action preserves $N$ and as a consequence the set of exceptions to GLC
\begin{equation}\label{exceptions}
\cE_d=\set{y\in Y_d : y\textrm{ is not L}}
\end{equation}
is $A_d$ invariant too. The main result in this paper is the following:
\begin{theorem}\label{theorem:the set of exceptions to Bigwood conjecture}
The set of exceptions to GLC, $\cE_d$, is a countable union of sets of upper box dimension $\le \dim A_d=d-1$.
\end{theorem}   
\noindent We remark that from the dimension point of view, this is the best possible result without actually proving GLC, because of the $A_d$ invariance. \\
In the fundamental paper \cite{EKL}, Eindiedler Katok and Lindenstrauss proved that the set of exceptions to Littlewood's conjecture is a countable union of sets of upper box dimension zero. The main tool in their proof is a deep measure classification theorem. The proof of theorem ~\ref{exceptions} is based on the same ideas and techniques
and further more goes along the lines of \cite{EK}. The new ingredient in the proof is lemma ~\ref{dimension and entropy}. As corollaries of this we get:
\begin{corollary}\label{corollary 1}
The set of $x\in X_d$ (resp $v\in\bR^d$) that are not GL, is a countable union of sets of upper box dimension $\le \dim A_d=d-1$. In particular, almost any lattice (resp vector) is GL.
\end{corollary}
\begin{corollary}[cf \cite{EKL} Theorem 1.5]\label{corollary 2}
For a fixed lattice $x\in X_d$, the set $\{y\in \pi^{-1}(x) : y \textrm{ is not L}\}$ is of upper box dimension zero. 
\end{corollary}
\begin{corollary}\label{corollary 3}
Any set in $X_d$ which has positive dimension transverse to the $A_d$ orbits, must contain a GL lattice. In particular if the dimension of the closure of an orbit $A_dx$ ($x\in X_d$) is bigger than $d-1$, then $x$ is GL.
\end{corollary}
We remark here that the only proof we know for the existence of GL lattices in dimension 2 and for GL vectors of any dimension, goes through the proof of theorem ~\ref{theorem:the set of exceptions to Bigwood conjecture}. For lattices of dimension $d\ge 3$ the situation is different. As will be seen in \S~\ref{sec:lattices that satisfy Bigwood}, for $d\ge 3$, one can exploit rigidity results on commuting hyperbolic toral automorphisms proved by Berend in \cite{B},
and give explicit constructions of GL lattice (and in fact of lattices which are \textit{GL of finite type}, see definition ~\ref{FLP}).\\
We end this section by noting that GLC is implied by a conjecture of G.A.Margulis (see ~\cite{Ma2}) which goes back to \cite{CaSD}.
\begin{conjecture}[Margulis]\label{Margulis conjecture}
Any bounded $A_{d+1}$ orbit in $X_{d+1}$ is compact.
\end{conjecture}
\noindent We prove in \S\S\ref{Linking dynamics to GLC}
\begin{proposition}\label{Margulis implies GLC}
Conjecture \ref{Margulis conjecture} implies GLC.
\end{proposition}
\textbf{Remark}: As will be seen, given a lattice $x\in X_d$, the richer the dynamics of its orbit under $A_d$, the easier it will be to prove that it is GL. This suggests an explanation to the difficulty of proving that $\bZ^2$ is GL (for it has a divergent orbit) as opposed to proving for example that a lattice with a compact or a dense orbit is GL.\\

\textbf{Acknowledgments}: I would like express my deepest gratitude to my advisers, Hillel Furstenberg and Barak Weiss for their constant help and encouragement. Special thanks are due to Manfred Einsiedler for his significant contribution to the results appearing in this paper. I would also like to express my gratitude to the mathematics department at the Ohio State University and to Manfred Einsiedler, for their warm hospitality during a visit in which much of the research was conducted. 
\section{Preparations}
\subsection{$X_d,Y_d$ as homogeneous spaces} Denote $G_d=SL_d(\bR), \Ga_d=SL_d(\bZ)$. We identify $X_d$ with the homogeneous space $G_d/\Ga_d$ in the following manner: For $g\in G_d$, the coset $g\Ga_d$ represents the lattice spanned by the columns of $g$. We denote this lattice by $\bar{g}$. $Y_d$ is identified with $\pa{G_d\ltimes \bR^d}/\pa{\Ga_d\ltimes \bZ^d}$ similarly, i.e. for $g\in G_d$ and $v\in \bR^d$, the coset $(g,v)\Ga_d\ltimes\bZ^d$ is identified with the grid $\bar{g} + v$. We endow $X_d,Y_d$ with the quotient topologies thus viewing them as homogeneous spaces. We define a natural embedding $\tau:Y_d\hookrightarrow X_{d+1}$ in the following manner: $\forall y=\bar{g}+v\in Y_d,$
\begin{equation}\label{def:tau}
 \tau_y=
\pa{
\begin{array}{cc}
g&v\\
0&1
\end{array}
}\Ga_{d+1}\in X_{d+1}.
\end{equation}
Note that this embedding is proper. $G_d$ and its subgroups act on $X_d,Y_d$. We embed $G_d$ in $G_{d+1}$ (in the upper left corner) thus allowing $G_d$ and its subgroups to act on $X_{d+1}$ as well. Note that the action commutes with $\tau.$
\subsection{Linking dynamics to GLC}\label{Linking dynamics to GLC}
The following observation is useful in connection with GLC: $\forall y\in Y_d$
\begin{equation}\label{another mu}
\inf_{n\ne 0} \av{nN(ny)}=\inf\set{\av{N(w)} : w\in\tau_y, w_{d+1}\ne 0 }.
\end{equation}
\begin{proposition}[Inheritance]\label{proposition 2}\quad
\begin{enumerate}
\item If $y,y_0\in Y_d$ are such that  $\tau_{y_0}\in\overline{A_{d+1}\tau_y}$, then if $y_0$ is L then so is $y$.
\item If $x,x_0\in X_d$ are such that  $x_0\in\overline{A_dx}$, then if $x_0$ is GL then so is $x$.
\end{enumerate} 
\end{proposition} 
\begin{proof}
The proof of (1) follows from \eqref{another mu}. The proof (2) follows from (1) and the compactness of the fibers of $\pi$. 
\end{proof}
As done in ~\cite{EKL} we link GLC to the dynamics of the following cone in $A_{d+1}$:
\begin{equation}\label{the cone A_{d+1}^+}
A_{d+1}^+=\left\{diag(e^{t_1}\dots e^{t_{d+1}})\in A_{d+1} : t_i> 0, i=1\dots d\right\}.
\end{equation}
\begin{lemma}\label{lemma:unboundedness imply LP}
For $y\in Y_d$, if $A_{d+1}^+\tau_y$ is unbounded (i.e has a noncompact closure), then $y$ is L. 
\end{lemma}
\begin{proof}
Recall Mahler's compactness criterion that says that a set $C\subset X_{d+1}$ is bounded, if and only if there is a uniform positive lower bound on the lengths of non zero vectors belonging to points of $C$. Let us fix the supremum norm on $\bR^d$ and $\bR^{d+1}$. Let $y\in Y_d$. Assume that in the orbit $A_{d+1}^+\tau_y$ there are lattices with arbitrarily short vectors.  Given $0<\eps<1$, there exists $a\in A_{d+1}^+$ and $w\in \tau_y$ such that the vector $aw$ is of length less than $\eps$. In particular $N(aw)=N(w)<\eps$. We will be through by \eqref{another mu} once we justify that $w_{d+1}\ne 0$. Assume $w_{d+1}=0$. It follows that the length of $aw$ is greater then that of $w$, as $A_{d+1}^+$ expands the first $d$ coordinates. On the other hand, the vector $w'=(w_1\dots w_d)^t\in\bR^d$ (which has the same length as $w$) belongs to the lattice $\pi(y)$.  Let $\ell$ denote the length of the shortest nonzero vector in $\pi(y)$. We obtain a contradiction once $\eps<\ell$.
\end{proof}
\begin{proof}[Proof of proposition \ref{Margulis implies GLC}]
Given a grid $y\in Y_d$, if $A_{d+1}^+\tau_y$ is unbounded then by lemma ~\ref{lemma:unboundedness imply LP} we know that $y$ is L. Assume that $A_{d+1}^+\tau_y\subset K$ for some compact $K\subset X_{d+1}$. Choose any one parameter semigroup $\left\{a_t\right\}_{t\ge 0}$ in the cone $A_{d+1}^+$ and let $z$ be a limit point of the trajectory $\left\{a_t\tau_y:t\ge 0\right\}$ in $K$. We claim that $z$ has a bounded $A_{d+1}$ orbit. To see this note that for any $a\in A_{d+1}$ we have that for large enough $t$'s, $aa_t$ is in the cone $A_{d+1}^+$, thus $az\in K.$ Assuming conjecture ~\ref{Margulis conjecture},  we obtain that $z$ has a compact $A_{d+1}$ orbit. It is easy to see that $\tau_y$ cannot be in the same $A_{d+1}$ orbit of $z$: $\tau_y$ has an unbounded $A_{d+1}$ orbit, for it contains vectors of the form $(*,\dots,*,0)^t$ which can be made as short as we wish under the action of $A_{d+1}$. By theorem 1.3 from ~\cite{LW}, there exist $1\le i,j\le d+1$ such that  $u_{ij}(t)z\in\overline{A_{d+1}\tau_y}$ for any $t\in \bR$, where  $u_{ij}(t)$, is the unipotent matrix all of whose entries are zero but the diagonal entries which are equal to 1 and  the $ij$'th entry that is equal to $t$. It is easy to see that for any $\eps>0$ there exist some $t$ such that $u_{ij}(t)z$ contains a vector $v$ with $N(v)\in(\eps,2\eps)$. Since $u_{ij}(t)z\in\overline{A_{d+1}\tau_y}$, we deduce that $\tau_y$ contains a vector $w$ with $N(w)\in(\eps,2\eps)$. We deduce that $w_{d+1}\ne 0$ and as $\eps$ was arbitrary, \eqref{another mu} implies that $y$ is L as desired. 
\end{proof}
\subsection{Dimension and entropy}
Let us recall the notions of \textrm{upper box dimension} and topological entropy.  Let $(X,d)$ be a compact metric space. For any $\epsilon>0$ we denote by $\cS_\epsilon=\cS_\eps(X)$ the maximum cardinality of a set of points in $X$ with the property that the distance between any pair of distinct points in it is greater or equal to $\epsilon$ (such a set is called $\epsilon -separated$). We define the upper box dimension of $X$ to be $$\dim_{box} X =\limsup_{\epsilon\to 0} \frac{\log \cS_\eps}{|\log\eps|}.$$ Since this is the only notion of dimension we will discuss, we shall denote it by $\dim X$. If we denote by $\cN_\eps=\cN_\eps(X)$ the minimum cardinality of a cover of $X$ by sets of diameter less than $\eps$, then we also have that 
$\dim X=\limsup_{\eps\to 0}\frac{\log \cN_\eps}{|\log\eps|}$.
Note that if $f:X\to Y$ is a bi-Lipschitz map, then $A\subset X$ is of zero dimension if and only if $f(A)\subset Y$ is.\\
If we have a continuous map $a:X\to X$, then for $\eps>0,n\in\bN$ we denote by $\cS_{n,\eps}=\cS_{n,\eps}(X,a)$ the maximum cardinality of a set $S\subset X$ with the property that for any pair of distinct points $x,y\in S$ there exist some $ 0\le i\le n$ with $d(a^ix,a^iy)>\eps$ (such a set is called $(n,\eps)-separated$ for $a$). The topological entropy of $a$ is defined to be 
$$h_{top}(a)=\lim_{\eps\to 0}\limsup_n \frac{\log\cS_{n,\eps}}{n}.$$  

\subsection{Metric conventions and a technical lemma}\label{Metric conventions}
For a metric space $(X,d)$ we denote by $B_\eps^X(p)$, the closed ball of radius $\eps$ around $p$. If $X$ is a group $B_\eps^X=B_\eps^X(p)$ where $p$ is the trivial element (zero or one according to the structure). 
Given Lie groups $G,H...$ we denote their Lie algebras  by the corresponding lower case Gothic letters $\gog,\goh..$. Let $G$ be a Lie group. We choose a right invariant metric $d(\cdot,\cdot)$ on it, coming from a right invariant Riemannian metric. Let $\Ga< G$ be a lattice in $G$. We denote the projection from $G$ to the quotient $X=G/\Ga$, by $g\mapsto\bar{g}$.  We define the following metric on $X$ (also denoted by $d(\cdot,\cdot)$)
\begin{equation}\label{eq:the metric on X_{d+1}}
d(\bar{g},\bar{h})=\inf_{\ga_i\in\Ga} d(g\ga_1,h\ga_2)=\inf_{\ga\in\Ga}d(g,h\ga).
\end{equation}
Under these metrics, for any compact set $K\subset X$ there exist an \textit{isometry radius} $\eps(K)$, i.e. a positive number $\eps$ such that for any $x\in K$, the map $g\mapsto gx$ is an isometry between $B_\eps^G$ and $B_\eps^X(x)$. 
Given a decomposition of $\gog=\oplus_1^l V_i$, the map $X\mapsto\exp v_1\dots\exp v_l$ (where $X=\sum v_i$ and $v_i\in V_i$) has the identity map as its derivative at zero. It follows that it is bi-Lipschitz on a ball of small enough radius around zero. We refer to such a map as a \textit{decomposition chart} and to the corresponding radius as a \textit{bi-Lipschitz radius}. When taking into account the above, we get that given a compact set $K\subset X$ and a decomposition $\gog=\oplus_1^lV_i$, we can speak of a \textit{bi-Lipschitz radius $\del(K)$, for $K$ with respect to this decomposition chart}, i.e. we choose $\del=\del(K)$ to be small enough so that the image of $B_\del^\gog$ under the decomposition chart will be contained in the ball of radius $\eps(K)$ around the identity element. Note that under these conventions a bi-Lipschitz radius for $K$ with respect to a decomposition chart is always an isometry radius.  We shall use the following notation: Given a semigroup $C\subset G$ and a compact set $K\subset X$ we denote 
\begin{equation}\label{K_C}
K_C=\left\{x\in X : Cx\subset K\right\}.
\end{equation} 
Note that $K_C$ is a compact (possibly empty) $C$-invariant set.\\ 
Let $G$ be semisimple and $\bR$-split (for our purpose it will be enough to consider $G=SL_d(\bR)$). Let $A<G$ be  a maximal $\bR$-split torus in $G$ (for example the group of diagonal matrices in $SL_d(\bR)$). We fix on $\gog$ a supremum norm with respect to a basis of $\gog$ whose elements belong to one dimensional common eigenspaces of the adjoint action of $A$.
For an element $a\in A$ we denote  by $U^\pm(a),\mathfrak{u}^\pm(a)$, the stable and unstable horospherical subgroups and Lie algebras associated with $a$. That is 
$$\begin{array}{ll}
U^+(a)=\left\{g\in G : a^{-n}ga^n\longrightarrow_{n\to\infty} e\right\},
U^-(a)=U^+(a^{-1}),\\
 \mathfrak{u}^\pm(a)=\left\{X\in \gog : Ad_a^n(X)\longrightarrow_{n\to\mp\infty} 0\right\}.
\end{array}$$
We denote by $\mathfrak{u}^0(a)$ the Lie algebra of the centralizer of $a$, that is $\{X\in\mathfrak{g} : Ad_a(X)=X\}.$ Note that from the semisimplicity of $a$, it follows that $\mathfrak{g}=\mathfrak{u}^+(a)\oplus\mathfrak{u}^0(a)\oplus\mathfrak{u}^-(a)$. When a fixed element $a\in A$ is given, we denote for $X\in\gog$, its components in $\gou^+,\gou^-,\gou^0$, by $X^+,X^-,X^0,$ respectively.  

We shall need the following lemma, the reader is advised to skip it for the time being and return to it after seeing it in use in the next section:
\begin{lemma}\label{technical lemma}
For a fixed element $e\ne a\in A$ 
 there exist $\lam>1$ and $\del,M,c>0$ such that for any $X_i,Y_i\in B_\del^{\gog}, i=1,2$ with $X_i\in\gou^+(a)$ and $||Y_1-Y_2||<\frac{||X_1-X_2||}{M}$, if for an integer $k$, for any $0\le j\le k$ 
$$d(a^j\exp X_1\exp Y_1,a^j\exp X_2\exp Y_2)<\del$$  then for any $0\le j\le k$  $$d(a^j\exp X_1\exp Y_1,a^j\exp X_2\exp Y_2)\ge c\lam^j||X_1-X_2||.$$
\end{lemma}  
\begin{proof}
Let $\eta>0$ be a bi-Lipschitz radius for the decomposition charts $\exp$ and $\vphi$, corresponding respectively to the trivial decomposition  and the decomposition $\gog=\gou^+(a)\oplus\gou^0(a)\oplus\gou^-(a)$ i.e. $\vphi:B_\eta^\gog\to G$ is the map $\vphi(v)=\exp v^+\exp v^0\exp v^-$. Let $0<\del_1<\eta$ satisfy 
\begin{equation}
\forall X_i,Y_i\in B_{\del_1}^\gog, i=1,2, \quad \exp X_1\exp Y_1\exp -Y_2\exp -X_2\in \vphi(B_\eta^\gog).
\end{equation} 
We can define $u: \left(B_{\del_1}^\gog\right)^4\to B_\eta^\gog$ by 
$$\forall X_i,Y_i\in B_{\del_1}^\gog, i=1,2\quad u(X_i,Y_i)=\vphi^{-1}\left(\exp X_1\exp Y_1\exp -Y_2\exp -X_2\right).$$
When $X_i,Y_i\in  B_{\del_1}^\gog, i=1,2$ are fixed, we simplify our notation and write instead of $u(X_i,Y_i),u^\pm(X_i,Y_i),u^0(X_i,Y_i)$, just $u,u^\pm,u^0$. 
Thus we have the identity: $\forall X_i,Y_i\in B_{\del_1}^\gog, i=1,2$
\begin{equation}\label{eq:X_i,Y_i}
 \vphi(u)=\exp u^+\exp u^0\exp u^-
=\exp X_1\exp Y_1\exp -Y_2\exp -X_2.
\end{equation}
Let us formulate two claims that we will use:\\
\quad\\
\textit{Claim 1}: There exist $0<\del_2<\del_1$ and $0<M,c_1$, such that   
$$\forall X_i,Y_i\in B_{\del_2}^\gog ,i=1,2,\textrm{ with } X_i\in\gou^+, ||Y_1-Y_2||<\frac{||X_1-X_2||}{M} $$
\begin{equation}\label{eq:u^+}
\textrm{we have } ||u^+||>c_1||X_1-X_2||.
\end{equation}
\quad\\
\textit{Claim 2}: There exist $0<\del_3<\del_2$, such that if $v\in B_\eta^\gog$ and $k\in\bN$ are such that $\forall 0\le j\le k, d\left(\vphi\left( Ad_a^j(v)\right),e\right)<\del_3$, then $\forall 0\le j\le k, Ad_a^j(v)\in B_\eta^\gog$.\\
\quad\\
Let us describe how to conclude the lemma from these claims: Let $\lam$ be the minimum amongst the absolute values of the eigenvalues of $Ad_a$ that are greater than $1$. Choose $\del=\del_3$ as in \textit{claim 2}, $M>0$ as in \textit{claim 1} and $c=c_1\cdot c_2$, where $c_1$ is as in \textit{claim 1} and $c_2$ satisfies $d(\vphi(v_1),\vphi(v_2))>c_2||v_1-v_2||$ for any $v_1,v_2\in B_\eta^\gog$. Note that because of the choice of the norm on $\gog$ , for $v\in\gog$ one has for any integer $j$
\begin{equation}\label{eq:norm and phi}
||Ad^j_a(v)||=||Ad^j_a(v^+)+Ad^j_a(v^0)+Ad^j_a(v^-)||\ge||Ad^j_a(v^+)||\ge\lam^j||v^+||.
\end{equation}
Let $X_i,Y_i$ and $k\in\bN$ be as in the statement of the lemma. For any $0\le j\le k$ we have:
$$\begin{array}{lll}
\del & > & d(a^j\exp X_1\exp Y_1,a^j\exp X_2\exp Y_2) \\
& =& d(a^j\exp X_1\exp Y_1\exp -Y_2\exp -X_2a^{-j},e)\\
& =& d(a^j\vphi(u(X_i,Y_i))a^{-j},e)\\
&= & d(\vphi(Ad_a^j(u)),e)\\
 & >& c_2||Ad_a^j(u)||\ge c_2\lam^j||u^+||
>c_1c_2\lam^j||X_1-X_2||=c\lam^j||X_1-X_2||.
\end{array}$$
We used the right invariance of the metric in the first equality, the fact that $\del<\del_1$ in the second and the relation $a\vphi(\cdot)a^{-1}=\vphi(Ad_a(\cdot))$ in the third. In the last row of inequalities we used \textit{claim 2} and the choice of $c_2$ in the first inequality, ~\eqref{eq:norm and phi} in the second and \textit{claim 1} in the third.
We now turn to the proofs of the above claims.
\begin{notation}
If two positive numbers $\al,\be$, satisfy $r\al<\be<\frac{1}{r}\al$, for some $r>0$, we denote it by $\al\sim_r\beta.$
\end{notation}
\textit{Proof of Claim 1.}
We use the notation of lemma ~\ref{technical lemma}. Let $0<\rho<\del_1$ be such that the map $(v_1,v_2)\mapsto \exp v_1\exp -v_2$, takes $\left(B_\rho^\gog\right)^2$ into $\exp B_\eta^\gog$. Since $\eta$ was chosen to be a bi-Lipschitz radius for the map $\exp$, there is a smooth function $w:\left(B_\rho^\gog\right)^2\to B_\eta^\gog$ which satisfies the relation
$$\forall v_1,v_2\in B_\rho^\gog, \quad \exp w(v_1,v_2)=\exp v_1\exp -v_2.$$
Note that if $v_1,v_2\in\gou^+$, then $w(v_1,v_2)\in B_\eta^{\gou^+}.$
Let $X_i,Y_i\in B_{\rho}^\gog, i=1,2$. The expressions in ~\eqref{eq:X_i,Y_i} are equal to
\begin{equation}\label{eq:1}
\exp w(X_1,X_2) \exp Ad_{\exp X_2}w(Y_1,Y_2).
\end{equation}
Let us sketch the line of proof we shall pursue: We show that $w(v_1,v_2)\sim_r||v_1-v_2||$, for some $r>0$. We choose the constants carefully in such a way that given $X_i,Y_i$ as in the statement of the claim, then there exist a $v\in\gog$ of length less than half of $||X_1-X_2||$, with $\vphi(v)= \exp Ad_{\exp X_2}w(Y_1,Y_2)$. It then follows from ~\eqref{eq:1} that  
$$\exp u^+=\exp w(X_1,X_2)\exp v^+=\exp w\left(w(X_1,X_2),-v^+\right).$$
It then follows that (ignoring constants that will appear) $$||u^+||=||w(X_1,X_2)+v^+||>||X_1-X_2||-||v^+||>\frac{||X_1-X_2||}{2}.$$
Let us turn now to the rigorous argument. 
The fact that $\eta$ is a bi-Lipschitz radius for $\exp$ implies the existence of a constant $r>0$, such that $\forall v_1,v_2\in B_\rho^\gog$
$$||v_1-v_2||\sim_r d(\exp v_1,\exp v_2)=d(\exp v_1\exp -v_2,e)\sim_r||w(v_1,v_2)||$$
\begin{equation}\label{eq:2}
\Rightarrow ||v_1-v_2||\sim_{r^2}||w(v_1,v_2)||.
\end{equation} 
Let $M_0$ bound from above the operator norm of $Ad_{\exp v}$ as $v$ ranges over $B_\rho^\gog$. In ~\eqref{eq:1}, we have
\begin{equation}\label{eq:3}
||Ad_{\exp X_2}w(Y_1,Y_2)||\le M_0||w(Y_1,Y_2)||<\frac{M_0}{r^2}||Y_1-Y_2||.
\end{equation}
Let $0<\del_2<\rho$ be such that $\frac{2\del_2}{r^2}<\rho$. This implies  by \eqref{eq:2}, that $\forall v_1,v_2\in B_{\del_2}^\gog, ||w(v_1,v_2)||<\rho.$ There exist some $0<\rho'<\eta$ such that $\exp(B_{\rho'}^\gog)\subset \vphi(B_\rho^\gog).$ Note that from the fact that $\exp$ is bi-Lipschitz on $B_{\rho'}^\gog$ and $\vphi^{-1}$ is bi-Lipschitz on $\exp\left(B_{\rho'}^\gog\right)$, it follows that there exist a constant $\tilde{r}$ such that 
\begin{equation}\label{eq:3.25}
\forall w\in B_{\rho'}^\gog, ||w||\sim_{\tilde{r}} ||\vphi^{-1}\left(\exp (w)\right)||.
\end{equation}
Let $M=\max\left\{\frac{2\del_2 M_0}{r^2\rho'},\frac{2M_0}{\tilde{r}r^4}\right\}.$ It follows from ~\eqref{eq:3}, that if $X_i,Y_i\in B_{\del_2}^\gog, i=1,2$ are such that $X_i\in\gou^+$ and $||Y_1-Y_2||<\frac{||X_1-X_2||}{M}$, then
\begin{equation}\label{eq:3.5}
||Ad_{\exp X_2}w(Y_1,Y_2)||<\frac{M_0||X_1-X_2||}{r^2M}
\end{equation}
and by our choice of $M$
\begin{equation}\label{eq:4}
||Ad_{\exp X_2}w(Y_1,Y_2)||<\rho'.
\end{equation}
By the choice of $\rho'$, there exist some $v\in B_\rho^\gog$ satisfying
\begin{equation}\label{eq:4.5}
\vphi(v)=\exp Ad_{\exp X_2}w(Y_1,Y_2).
\end{equation}
Note that by ~\eqref{eq:3.25} with $w=Ad_{\exp X_2}w(Y_1,Y_2)$, by ~\eqref{eq:3.5}, and by the choice of $M$
\begin{equation}\label{eq:5.5}
||v||\sim_{\tilde{r}}||Ad_{\exp X_2}w(Y_1,Y_2)||\Rightarrow ||v||<\frac{M_0||X_1-X_2||}{\tilde{r}r^2M}\le \frac{r^2||X_1-X_2||}{2}.
\end{equation} 
The expressions in ~\eqref{eq:X_i,Y_i} and in ~\eqref{eq:1} are equal to
\begin{equation}\label{eq:5}
\exp u^+\exp u^0\exp u^-=\exp w(X_1,X_2)\exp v^+\exp v^0\exp v^-.
\end{equation}
As remarked above, the fact that $X_i\in\gou^+$ implies that $w(X_1,X_2)\in \gou^+.$
From our choice of $\rho$ and the fact that $v^+,w(X_1,X_2)\in B_\rho^{\gou^+}$, it follows that $w\left(w(X_1,X_2),-v^+\right)\in \gou^+$ is defined and satisfies: 
\begin{equation}\label{eq:6}
\exp w(X_1,X_2)\exp v^+=\exp w\left(w(X_1,X_2),-v^+\right).
\end{equation}
Because $w(\cdot,\cdot)$ takes values in $B_\eta^\gog$, and $\vphi$ is injective on $B_\eta^\gog$, it follows from ~\eqref{eq:5},~\eqref{eq:6}, that 
$$u^+=w\left(w(X_1,X_2),-v^+\right).$$
Because of ~\eqref{eq:2} and ~\eqref{eq:5.5}
$$||u^+||=||w\left(w(X_1,X_2),-v^+\right)||>r^2||w(X_1,X_2)+v^+||$$
$$\ge r^2\left(||w(X_1,X_2)||-||v^+||\right)\ge r^4||X_1-X_2||-\frac{r^4}{2}||X_1-X_2||=\frac{r^4}{2}||X_1-X_2||.$$
Thus \textit{claim 1} follows with the above choices of $\del_2$ and $M$ and with $c_1=\frac{r^4}{2}.$\\
\textit{Proof of Claim 2.} Let $M_1=||Ad_a||$. Let $0<\del_3<\del_2$ satisfy 
\begin{equation}\label{eq:claim2}
B_{\del_3}^G\subset\vphi\left( B_{\frac{\eta}{M_1}}^\gog\right).
\end{equation}
Let $v\in B_\eta^\gog$ and $k\in\bN$ satisfy the assumptions of \textit{claim 2}. Assume by way of contradiction that there exist some $0\le j <k$ such that $Ad_a^j(v)\in B_\eta^\gog$ but $Ad_a^{j+1}(v)\notin B_\eta^\gog$. We conclude that $$\eta<M_1||Ad_a^j(v)||\Rightarrow Ad_a^j(v)\notin B_{\frac{\eta}{M_1}}^\gog.$$    
This contradicts the assumption that $\vphi\left(Ad_a^j(v)\right)\in B_{\del_3}^G$ and ~\eqref{eq:claim2} because $\vphi$ is injective on $B_\eta^\gog$. 
\end{proof}
\section{The set of exceptions to GLC}
\label{sec:the set of exceptions}
In this section we prove theorem \ref{theorem:the set of exceptions to Bigwood conjecture} and its corollaries. We go along the lines of \S 4 in ~\cite{EK} and the fundamental ideas appearing in \cite{EKL}. The main hidden tool is the measure classification theorem in \cite{EKL}. The main reason which prevents us from citing known results is the fact that in the embedding $\tau: Y_d\hookrightarrow X_{d+1}$ ~\eqref{def:tau}, the grids which are not L and thus have bounded $A^+_{d+1}$ orbit, do not lie (locally) on single unstable leaves of elements in the cone, but lie on products of unstable leaves (see lemmas ~\ref{lemma 4.2 from EK},~\ref{dimension and entropy}). 
The following two lemmas furnish the link between dimension and entropy in our discussion. Lemma ~\ref{lemma 4.2 from EK} is essentially lemma 4.2 from \cite{EK}. For the reader's convenience and the completeness of our presentation, we include the proof in the appendix. Lemma ~\ref{dimension and entropy} is one of the new ingredients appearing in this paper (recall the notation of ~\eqref{K_C}):
\begin{lemma}\label{lemma 4.2 from EK}
Let $C\subset A_{d+1}$ be a semigroup, $a\in C$ and $K\subset X_{d+1}$ a compact set. If for some $\del>0$ and $x\in K$
$$\dim\left(\exp \left(B_\del^{\mathfrak{u}^+(a)}\right)\cdot x\cap K_C\right)>0$$
then $a$ acts with positive topological entropy on $K_C$.
\end{lemma}
\noindent For the proof of theorem ~\ref{theorem:the set of exceptions to Bigwood conjecture} we shall need the following generalization of lemma \ref{lemma 4.2 from EK}:
\begin{lemma}\label{dimension and entropy}
Let $C_2\subset C_1\subset A_{d+1}$ be semigroups, $a_i\in C_i,  i=1,2,$ and $K\subset X_{d+1}$ a compact set. Assume that there exists subspaces $V_i$ of 
$\mathfrak{u}^+(a_i)$ such that for any $b\in C_2$, $V_1\subset \mathfrak{u}^-(b)$. Then, there exists $\del>0,$ such that if for some $x\in K$  
$$\dim\left(\exp B_\del^{V_1} \exp B_\del^{V_2}\cdot x\cap K_{C_1}\right)>0,$$
then either $a_1$ acts with positive topological entropy on $K_{C_1}$, or there exists a compact set $\tilde{K}\supset K$, such that $a_2$ acts with positive topological entropy on $\tilde{K}_{C_2}$.
\end{lemma}
\noindent The following corollary goes along the lines of Proposition 4.1 from \cite{EK}.
\begin{corollary}\label{dimension=0}
If in lemma \ref{dimension and entropy}, we assume that the $C_i$'s are open cones in $A_{d+1}$, then for any $x\in K$
$$\dim\left(\exp B_\del^{V_1} \exp B_\del^{V_2}\cdot x\cap K_{C_1}\right)=0.$$
\end{corollary}
\begin{proof}
In the proof of Proposition 4.1 from ~\cite{EK}, it is shown that there cannot be an open cone $C\subset A_{d+1}$ that acts on a compact invariant subset of $X_{d+1}$ such that some element in $C$ acts with positive topological entropy. Thus by lemma ~\ref{dimension and entropy}, positivity of the dimension leads to a contradiction.  
\end{proof}
\noindent \textbf{Remark}: The highly non trivial part of the proof of theorem ~\ref{theorem:the set of exceptions to Bigwood conjecture} is hidden in the proof of corollary ~\ref{dimension=0}. This is the use of the measure classification from ~\cite{EKL}. 
\begin{proof}[Proof of theorem ~\ref{theorem:the set of exceptions to Bigwood conjecture}]
As the embedding of $Y_d$ in $X_{d+1}$ \eqref{def:tau} is bi-Lipschitz, lemma ~\ref{lemma:unboundedness imply LP}  will imply the theorem once we show that for any compact $L\subset \tau(Y_d)$ 
\begin{equation}\label{eq:1.31}
\dim\left(L_{A_{d+1}^+}\right)\le d-1.
\end{equation}
We use lemma ~\ref{dimension and entropy} and corollary ~\ref{dimension=0} with the following choices: We take $$C_1=A_{d+1}^+,\quad 
a_1=diag\left(2,\dots ,d+1,1/(d+1)!\right),\quad a_2=diag\left(d+1,\dots ,2,1/(d+1)!\right)$$
$$V_1=\left\{
 \left(
\begin{array}{lllll}
0&\dots&&&0\\
\star &\ddots &&&\vdots\\
\vdots &\ddots & & &\\
 \star&\dots&\star &\ddots &\\
 0&\dots& &0&0
 \end{array}\right)\in\gog_{d+1}\right\}, V_2=\mathfrak{u}(a_2)^+=\left\{
\left(
\begin{array}{llll}
0&\star&\dots&\star\\
\vdots &\ddots & \ddots &\vdots\\
 & & & \star\\
 0&\dots & &0
 \end{array}\right)\in\gog_{d+1}\right\}.$$
Note that indeed $V_1\subset \mathfrak{u}^-(a_2)$, thus we choose $C_2$ to be an open cone containing $a_2$ and contained in $C_1$, such that for any $b\in C_2, V_1\subset\mathfrak{u}^-(b)$. Finally we take $K=\Om L$, where $\Om$ is a symmetric compact neighborhood of the identity in $A_d$. Corollary  ~\ref{dimension=0} now tells us that there exists $\del>0,$ such that for any $x\in K$ 
\begin{equation}\label{eq:1.32}
\dim\left(\exp B_\del^{V_1} \exp B_\del^{V_2}\cdot x\cap K_{C_1}\right)=0.
\end{equation} 
Note that for any $x\in L$, $\Om\exp B_\del^{V_1}\exp B^{V_2}_\del x$ is a compact neighborhood of $x$ in $\tau(Y_d)$. Cover  $L$ by finitely many sets of the form $\Om\exp B_\del^{V_1}\exp B^{V_2}_\del x_i$, for a suitable choice of points $x_i\in L$. It follows that $L_{C_1}$ is a finite union of sets of the form $\left(\Om\exp B_\del^{V_1}\exp B^{V_2}_\del x_i\right) \cap L_{C_1}.$ Thus ~\eqref{eq:1.31} will follow from ~\eqref{eq:1.32}, once we show that for any $i$ 
\begin{equation}\label{eq:1.3 1}
 \left(\Om\exp B_\del^{V_1}\exp B^{V_2}_\del x_i\right) \cap L_{C_1}\subset\Om\left(\exp B_\del^{V_1}\exp B^{V_2}_\del x_i\cap K_{C_1}\right).
\end{equation} 
This is true because if $y=a\exp X\exp Y x_i$ is an element of the left hand side of ~\eqref{eq:1.3 1}, then  $C_1y\subset L$ and if we define $y'=a^{-1}y=\exp X\exp Y x_i$, then $C_1y'\subset a^{-1}L\subset K$ and so $ay'=y$ is an element of the right hand side of ~\eqref{eq:1.3 1}. 
\end{proof}
\begin{proof}[Proof of lemma ~\ref{dimension and entropy}] 
Note that from the fact that $V_1\subset\mathfrak{u}^-(a_2)$ it follows that the sum $V_1+V_2$ is direct. Let $V_3$ be any subspace of $\gog_{d+1}$ such that $\gog_{d+1}=V_1\oplus V_2 \oplus V_3$. Choose $\del=\del(K)$ to be a bi-Lipschitz radius for $K$ with respect to the above decomposition (see \S\S~\ref{Metric conventions} for notation). Assume also that $\del$ satisfies the conclusion of lemma ~\ref{technical lemma} with $a=a_1$. For the sake of brevity we denote $B_i=B_\del^{V_i}$. Assume $x\in K$ satisfies
\begin{equation}\label{eq:3.21}
\dim\left(\exp B_1\exp B_2 x\cap K_{C_1}\right)>0.
\end{equation}
Since $\del$ is a bi-Lipschitz radius for $K$ (with respect to the decomposition $V_1\oplus V_2\oplus V_3$), ~\eqref{eq:3.21} implies that the dimension of 
$$F=F(\del)=\left\{(X,Y)\in B_1\times B_2 : \exp X\exp Y x\in K_{C_1}\right\}$$
is positive. From the choice of the norm on $\gog_{d+1}$ (see \S\S~\ref{Metric conventions}) and from the assumption that for any $b\in C_2$, $V_1\subset \mathfrak{u}^-(b),$  it follows that for any $X\in V_1$, $||Ad_b(X)||\le||X||$. Choose a compact set $\tilde{K}\supset \exp(B_{\del}^{\gog_{d+1}}) K$. Denote by $\pi_2$ the projection from $B_1\times B_2$ to $B_2$. There are two cases:\\
\textbf{Case 1}: Assume $\dim\pi_2(F)>0$. We claim that $\exp\left(\pi_2(F)\right)x\subset\tilde{K}_{C_2}$. To see this, note that if $Y\in \pi_2(F)$ then there exists some $X\in B_1$ such that $\exp X\exp Y x\in K_{C_1}$, and so for any $b\in C_2$ we have
$$b\exp Yx=\exp Ad_b(-X)b\exp X\exp Yx\in \exp (B_{\del}^{\gog_{d+1}}) K\subset \tilde{K}.$$
Now $\exp\left(\pi_2(F)\right)x\subset\exp B_2\cdot x\cap \tilde{K}_{C_2}$ and therefore, positivity of the dimension of $\pi_2(F)$, implies the positivity of the dimension of $\exp B_2\cdot x\cap \tilde{K}_{C_2}$. We apply lemma ~\ref{lemma 4.2 from EK} and conclude that $a_2$ acts with positive topological entropy on $\tilde{K}_{C_2}$.\\
\textbf{Case 2}: Assume $\dim\pi_2(F)=0$ and let us denote $\dim F= 3\rho$ with $\rho>0$. We will show that $a_1$ acts with positive topological entropy on $K_{C_1}$. 
Recall the notation of lemma ~\ref{technical lemma} (applied to $a_1$). 
We shall find for arbitrarily large integers $n$, sets $S_n\subset F$ with the following properties:
\begin{itemize}
\item For any pair of distinct points $(X_i,Y_i)\in S_n, i=1,2$ 
\begin{equation}\label{eq:two points}
||X_1-X_2||>\lam^{-n}, ||Y_1-Y_2||<\frac{\lam^{-n}}{M}.
\end{equation} 
\item $|S_n|> M^{-\rho}\lam^{{n}\rho}.$
\end{itemize} 
Given two distinct points in $S_n$, $(X_i,Y_i), i=1,2$, let us analyze the rate at which $\exp X_i\exp Y_i x$ drift apart from each other under the action of powers of $a_1$.
For any $j\ge 0$ we have that $a_1^j\exp X_i\exp Y_i x\in K$ by the definition of $F$, and so, if the distance between these two points is less than $\del$ (which is also an isometry radius for $K$), we have 
\begin{equation}\label{eq:drifting}
\begin{array}{l}
d(a_1^j\exp X_1\exp Y_1 x,a_1^j\exp X_2\exp Y_2 x)= \\
d(a_1^j\exp X_1\exp Y_1 ,a_1^j\exp X_2\exp Y_2 ) .  
\end{array} 
\end{equation}
By lemma \ref{technical lemma}, for any $k$ such that for all $0\le j\le k$, the expressions in ~\eqref{eq:drifting} are smaller than $\del$, we have 
$$d(a_1^j\exp X_1\exp Y_1x,a_1^j\exp X_2\exp Y_2)\ge c\lam^k||X_1-X_2||>c\lam^{k-n}.$$
In particular, if we set $\eps_0= \min\{c,\del\}$ then we must have some $0\le j\le n$ for which 
$$d(a_1^j\exp X_1\exp Y_1x,a_1^j\exp X_2\exp Y_2x)>\eps_0.$$
This means that $\left\{\exp X\exp Y x : (X,Y)\in S_n\right\}$ is an $(n,\eps_0)$-separated set for $(K_{C_1},a_1)$. From here, it is easy to derive the positivity of the entropy by the bound we have on the size of $S_n$: $$h_{top}(K_{C_1},a_1)\ge\limsup\frac{1}{n}\log|S_n|\ge\lim_n \frac{1}{n}\log (M^{-\rho}\lam^{{n}\rho})=\rho\log \lam>0.$$
\quad \\
To build the sets $S_n$ with the above properties, for arbitrarily large $n$'s, we argue as follows: By definition of the dimension one can find a sequence $\eps_k\searrow 0$ such that $$\cS_{\eps_k}(F)>(1/\eps_k)^{2\rho}.$$
Choose $n_k\nearrow\infty$ such that $\lam^{-n_k}\le\eps_k<\lam^{-n_k+1}$. It follows that 
\begin{equation}\label{eq:separating estimation}
\cS_{\lam^{-n_k}}(F)\ge \cS_{\eps_k}(F)>\lam^{2n_k\rho-2\rho}.
\end{equation}
On the other hand, because we assume $\dim\pi_2(F)=0$, for any large enough $n$ $$\frac{\log(\cN_{\frac{\lam^{-n}}{M}}(\pi_2(F)))}{\log(\lam^nM)}<\rho.$$ Hence  
\begin{equation}\label{eq:s(pi(F))}
\cN_{\frac{\lam^{-n}}{M}}(\pi_2(F))<\lam^{n\rho}M^\rho.
\end{equation}
Denote $N_k=\cN_{\frac{\lam^{-n_k}}{M}}(\pi_2(F))$ and let $E_i^{(k)}, i=1\dots N_k$ be a covering of $\pi_2(F)$ by subsets of $B_2$ of diameter less than $\frac{\lam^{-n_k}}{M}$. Since $N_k<\lam^{n_k\rho}M^\rho$, by \eqref{eq:separating estimation} and the pigeon hole principle, there must exist some $1\le i_k\le N_k$ with
\begin{equation}\label{eq:s(pi^{-1})}
\cS_{\lam^{-n_k}}(\pi_2^{-1}(E^{(k)}_{i_k})\cap F)>\lam^{{n_k}\rho}M^{-\rho}.
\end{equation}
Define $S_{n_k}$ to be a maximal $\lam^{-n_k}$-separated set in $\pi_2^{-1}(E^{(k)}_{i_k})\cap F$. By construction, $S_{n_k}$ has the desired properties. 
\end{proof}
\begin{proof}[Proof of corollary ~\ref{corollary 1}] The projection $\pi:Y_d\to X_d$ cannot increase dimension and therefor $\pi(\cE_d)$, the set of lattices which are not GL is a countable union of sets of upper box dimension $\le d-1$.\\
Denote by $p:X_d\times \bR^d\to Y_d$ the map $p(x,v)=x+v$. It is bi-Lipschitz with a countable fiber and so if we denote by $p_2:X_d\times \bR^d\to \bR^d$ the natural projection, then $p_2(p^{-1}(\cE_d))$, the set of vectors which are not GL, is a countable union of sets of upper box dimension $\le d-1$.
\end{proof}
\begin{proof}[Proof of corollary ~\ref{corollary 2}] Assume by way of contradiction that there exist $x\in X_d$ with $$\dim\pa{\pi^{-1}(x)\cap \cE_d}>0.$$  It follows that if $\Om \subset A_d$ is a compact neighborhood of the identity, then $$\Om\pa{\pi^{-1}(x)\cap \cE_d}>0 ,$$ has dimension greater than $\dim A_d=d-1$. A contradiction to theorem ~\ref{theorem:the set of exceptions to Bigwood conjecture}.
\end{proof}
\begin{proof}[Proof of corollary ~\ref{corollary 3}] 
Positivity of the dimension of a subset $L\subset X_d$, transverse to the $A_d$ orbits, means that $A_dL$ contains a compact set of dimension greater then $\dim A_d=d-1$. By theorem  ~\ref{theorem:the set of exceptions to Bigwood conjecture}, such a set must contain a GL lattice. If the dimension of the $A_d$ orbit closure of a lattice $x\in X_d$ is greater then $d-1$, then it contains a GL lattice. It now follows from proposition ~\ref{proposition 2}, that $x$ is GL.
\end{proof}
\section{Lattices that satisfy GLC}
\label{sec:lattices that satisfy Bigwood}
In this section we shall explicitly build L lattices in $\bR^d$ for $d\ge 3$. In fact these lattices will possess a much stronger property, namely:
\begin{definition}\label{FLP}
A grid $y\in Y_d$ is FL (L of finite type) if there exist a non zero integer $n$ such that $N(ny)=0$. A lattice $x\in X_d$ is GFL if any $y\in\pi^{-1}(x)$ is FL.
\end{definition}
\begin{definition}\label{rational grids}
A grid $y\in Y_d$ is rational if $y$ is a torsion point in $\pi^{-1}(\pi(y))$.
\end{definition}
The following list of observations is left to be verified by the reader. 
\begin{proposition}\label{observation}
\begin{enumerate}
\item The set of FL grids is $A_d$ invariant.
\item If $y,y_0\in Y_d$, $y_0\in \overline{A_dy}$ and $y_0$ is FL then $y$ is FL too.
\item If $x,x_0\in X_d$, $x_0\in \overline{A_dx}$ and $x_0$ is GFL then $x$ is GFL too.
\item Any rational grid is FL.
\item If $x_1\subset X_{d_1}$ is GFL and $x_2\in X_{d_2}$ is any lattice, then $x_1\oplus x_2\in X_{d_1+d_2}$ is GFL.
\item If $x_1,x_2\in X_d$ are such that $x_1$ is GFL and there exist some $c>0$ such that $cx_1$ is commensurable with $x_2$ then $x_2$ is GFL.
\item The standard lattice $\bZ^d$ is not GFL. In fact, for any vector $v\in\bR^d$, all of whose coordinates are irrationals we have that $\bZ^d+v$ is not FL.
\end{enumerate} 
\end{proposition}
For $x\in X_d$, denote by $A_{d,x}$ its stabilizer in $A_d$. Note that $A_{d,x}$ acts on the torus $\pi^{-1}(x)$ as a group of automorphisms. 
From (2) and (4) of proposition ~\ref{observation}, we deduce the following
\begin{lemma}\label{stabilizer}
If for any grid $y\in \pi^{-1}(x)$, $\overline{A_{d,x}y}\subset \pi^{-1}(x)$ contains an FL grid then $x$ is GFL. In particular, if for any grid $y\in \pi^{-1}(x)$, $\overline{A_{d,x}y}$ contains a rational grid then $x$ is GFL.
\end{lemma}
Recall that a group of automorphisms of a torus $\pi^{-1}(x)$ ($x\in X_d$) is called ID, if any infinite invariant set is dense. The following is a weak version of theorem 2.1 from ~\cite{B}:
\begin{theorem}[Theorem 2.1 \cite{B}]\label{ID}
If the stabilizer $A_{d,x}$ of  a lattice $x\in X_d$ under the action of $A_d$ satisfies
\begin{enumerate}
\item There exist some $a\in A_{d,x}$ such that for any $n$ the characteristic polynomial of $a^n$ (which is necessarily over $\bQ$) is irreducible.
\item For each $1\le i\le d$ there exist $a=diag(a_1\dots a_d)\in A_{d,x}$ with $a_i\ne 1$.
\item There exists $a_1,a_2\in A_{d,x}$ which are multiplicatively independent (that is $a_1^ka_2^m=1\Rightarrow k=m=0$). 
\end{enumerate} 
then $A_{d,x}$ is ID.
\end{theorem}
We now turn to the construction of a family of GFL lattices. Let $K$ be a totally real number field of degree $d$ over $\bQ$.
\begin{definition}
\begin{enumerate}
\item A lattice in $K$ is the $\bZ$-span of a basis of $K$ over $\bQ$.
\item If $\Lam$ is a lattice in $K$ then its associated order is defined as $O_\Lam=\left\{x\in K : x\Lam\subset \Lam\right\}.$ 
\end{enumerate}
\end{definition}
It can be easily verified that for any lattice $\Lam$ in $K$, $O_\Lam$ is a ring. Moreover, the units in this ring are exactly $O_\Lam^*=\left\{\om\in K : \om\Lam=\Lam\right\}$. Dirichlet's unit theorem states the following
\begin{theorem}[Dirichlet's unit theorem]
For any lattice $\Lam$ in $K$, the group of units $O_\Lam^*$ is isomorphic to $\{\pm1\}\times\bZ^{d-1}.$ 
\end{theorem}
Let $\sig_1\dots\sig_d$ be an ordering of the different embeddings of $K$ into the reals. Define $\vphi:K\to\bR^d$ to be the map whose $i$'th coordinate is $\sig_i$. If we endow $\bR^d$ with the structure of an algebra (multiplication defined coordinatewise), then $\vphi$ becomes a homomorphism of $\bQ$ algebras (here we think of the fields $\bQ,\bR$ as embedded diagonally in $\bR^d$). It is well known that if $\Lam$ is a lattice in $K$, then $\vphi(\Lam)$ is a lattice in $\bR^d$. Let us denote by $x_\Lam$ the point in $X_d$ obtained by normalizing the covolume of $\vphi(\Lam)$ to be $1$. We refer to such a lattice as a \textit{lattice coming from a number field}.  Because $\vphi$ is a homomorphism
$$\vphi(O_\Lam^*)\subset \left\{a\in\bR^d : ax_\Lam=x_\Lam\right\}.$$
We can identify the linear map obtained by left multiplication by $a\in\bR^d$ on $\bR^d$ with the usual action of the diagonal matrix whose entries on the diagonal are the coordinates of $a$. We abuse notation and denote the corresponding matrix by the same symbol. After recalling that the product of all the different embeddings of a unit in an order equals $\pm1$ we get that in fact
$\vphi(O_\Lam^*)$ is a subgroup of the stabilizer of $x_\Lam$ in the group of diagonal matrices  of determinant $\pm 1$ (in fact there is equality here but we will not use it). To get back into $SL_d$ we replace $O_\Lam^*$ by the subgroup $O_{\Lam,+}^*$ of totally positive units (that is units, all of whose embeddings are positive). It is a subgroup of finite index in $O_\Lam^*$. We conclude that $\vphi$ will map $O_{\Lam,+}^*$ into $A_{d,x_\Lam}$ (using our identification of vectors and diagonal matrices). 
\begin{lemma}\label{lemma:g_Lam is ID}
If $x_\Lam\in X_d$ is a lattice coming from a totally real number field $K$ of degree $d\ge3$, then $A_{d,x_\Lam}$ is an ID group of automorphisms of $\pi^{-1}(x_\Lam)$.  
\end{lemma}
\begin{proof}
It is enough to check that conditions (1),(2),(3) from theorem \ref{ID} are satisfied. Condition (2) is trivial. Condition (3) is a consequence of Dirichlet's units theorem and the assumption $d\ge 3$. To verify condition (1) we argue as follows: We will show that there exist $\al\in O_K^*$ such that for any $n$, $\al^n$ generates $K$ (this is enough because $O_{\Lam,+}^*$ is of finite index in $O_K^*$). Let $F_1\dots F_k$ be a list of all the subfields of $K$. If we denote for a subset $B\subset K$ $$\sqrt{B}=\left\{x\in K : \exists n\textrm{ such that } x^n\in B\right\}$$ then we need to show that \begin{equation}\label{O_F}
O_K^*\setminus \cup_1^k\sqrt{O_{F_i}^*}\ne\emptyset.
\end{equation}
Fix a proper subfield $F$ of $K$. Note that the following is an inclusion of groups $O_F^*\subset \sqrt{O_F^*}\subset O_K^*$. Thus, Dirichlet's units theorem will imply ~\eqref{O_F} once we prove that $O_F^*$ is of finite index in $\sqrt{O_F^*}$. We shall give a bound on the order of elements in the quotient $\sqrt{O_F^*}/O_F^*$ thus showing that the groups are of the same rank. It is enough to show that there exist some integer $n_0$ such that if $x\in K$ satisfies $x^n\in F$ for some $n$ then $x^{n_0}\in F$. Let $x\in K$ be such an  element. Denote by $\sig_1\dots\sig_r$ the different embeddings of $F$ into the reals and for any $1\le i\le r$, denote by $\sig_{ij}, j=1\dots s$ the different extensions of $\sig_i$ to an embedding of $K$ into the reals. Thus $d=rs$ and $\sig_{ij}$ are all the different embeddings of $K$ into the reals. Note that $x^n\in F$ if and only if for any $1\le i\le r$ $\sig_{i1}(x^n)=\dots =\sig_{is}(x^n)$ i.e. if and only if $(\frac{\sig_{ij}(x)}{\sig_{ik}(x)})^n=1$ for all $i,j,k$. But since there is a bound on the order of roots of unity in $K$ we are done. 
\end{proof}
We are now in position to prove
\begin{theorem}\label{FL}
Any lattice coming from a totally real number field of degree $d\ge3$ is GFL.
\end{theorem} 
\begin{proof}
Let $x_\Lam\in X_d$ be a lattice coming from a totally real number field of degree $d\ge 3$. Using lemma \ref{lemma:g_Lam is ID} and lemma \ref{stabilizer} we see that the theorem will follow if we will show that any finite $A_{d,x_\Lam}$ invariant set in $\pi^{-1}(x_\Lam)$ contain only rational grids. Assume that $y\in\pi^{-1}(x_\Lam)$ lies in a finite invariant set. It follows that there exist $e\ne a\in \vphi(O^*_{\Lam,+})$ with 
\begin{equation}\label{stabilize y}
ay=y.
\end{equation}
Write $x_\Lam=c\vphi(\Lam)$, $y=x_\Lam + v$ and $a=\vphi(\om)$. Then from \eqref{stabilize y} it follows that there exist $\theta\in\Lam$ such that in the algebra $\bR^d$ $$v(\vphi(\om)-1)=c\vphi(\theta)\Rightarrow v=c\vphi(\theta(\om-1)^{-1}).$$
Since $K$ is spanned over $\bQ$ by $\Lam$ we see that $v$ is in the $\bQ$ span of $c\vphi(\Lam)=x_\Lam$ and hence $y$ is a rational grid as desired.
\end{proof}
As a corollary of the ergodicity of the $A_d$ action on $X_d$ and proposition \ref{observation} (3), we get the following (we refer the reader to ~\cite{Sh} for a stronger result).
\begin{corollary}
Almost any $x\in X_d$ is GFL for $d\ge 3$.
\end{corollary}
The following result appears for example in \cite{LW}: 
\begin{theorem}
The compact orbits for $A_d$ in $X_d$ are exactly the orbits of lattices coming from totally real number fields of degree $d$. 
\end{theorem}   
This gives us the following corollary, which is a special case of proposition \ref{observation} (3), combined with theorem ~\ref{FL}. We state it separately because of its interesting resemblance to theorem 1.3 from \cite{LW}.
\begin{corollary}
For $x\in X_d$, if $\overline{A_dx}$ contains a compact $A_d$ orbit, then $x$ is GFL.
\end{corollary}
Let us end this paper with two open problems which emerge from our discussion. 
\begin{problem}
Give an explicit example of a Littlewood lattice in dimension 2. In particular, prove that any lattice with a compact $A_2$ orbit, is L. 
\end{problem}
\begin{problem}
Do there exists two dimensional GFL lattices? 
\end{problem}
\section{Appendix}
\begin{proof}[Proof of lemma ~\ref{lemma 4.2 from EK}]
Let the notation be as in lemma ~\ref{lemma 4.2 from EK}. The statement of lemma \ref{technical lemma} simplifies when one chooses the $Y_i$'s  to be zero in the original statement:\\
\textbf{Lemma \ref{technical lemma}, simplified version}:\textit{
For a fixed element $e\ne a\in A_{d+1}$ 
there exist $\lam>1$ and $\eta,c>0$ such that for any $X_i\in B_\eta^{\gou^+(a)}, i=1,2$, if for an integer $k$, for any $0\le j\le k$ 
$$d(a^j\exp X_1,a^j\exp X_2)<\eta$$  then for any $0\le j\le k$  $$d(a^j\exp X_1,a^j\exp X_2)\ge c\lam^j||X_1-X_2||.$$}\\
We apply this lemma for the element $a\in A_{d+1}$ appearing in the statement of lemma ~\ref{lemma 4.2 from EK}. Let $0<\del'<\max\left\{\eta,\del\right\}$ be a bi-Lipschitz radius for $K$, with respect to the chart $\exp$ (see \S\S~\ref{Metric conventions} for notation). Cover the compact set $\exp B_\del^{\gou^+(a)}x\cap K_C$ by finitely many sets of the form $\exp B_{\del'}^{\gou^+(a)}y_i\cap K_C$, for a suitable choice of points $y_i\in K_C$. By assumption there exists an $i$ such that $\dim\left(\exp B_{\del'}^{\gou^+(a)}y_i\cap K_C\right)>0.$ Because $\del'$ is a bi-Lipschitz radius, we have that the dimension of 
\begin{equation}\label{eq:5.1}
F=F(\del')=\left\{X\in B_{\del'}^{\gou^+(a)} : \exp Xy_i\in K_C\right\}
\end{equation}
is positive. Denote it by $2\rho$. By definition of dimension this means that there exists a sequence $\eps_k\searrow 0$, and $\eps_k-separated$ sets $S_k\subset F$, such that $|S_k|>\eps_k^{-\rho}.$ Let $n_k\nearrow\infty$ be a sequence such that $\lam^{-n_k}\le \eps_k<\lam^{-n_k+1}.$ 
Let $X_1,X_2\in S_k$ be two distinct points. Because $\del'$ is also an isometry radius for $K$, if $\ell$ is an integer such that $\forall 0\le j\le \ell$, $d(a^j\exp X_1y_i,a^j\exp X_2y_i)<\del'$, then the simplified version of lemma ~\ref{technical lemma}, stated above implies that $\forall 0\le j\le \ell$
\begin{equation}\label{eq:5.2}
\begin{array}{ll}
\del' & >d\left(a^j\exp X_1y_i,a^j\exp X_2y_i\right)\\
      & =d\left(a^j\exp X_1,a^j\exp X_2\right)\\
      & \ge c\lam^j||X_1-X_2||\ge c\lam^j\eps_k>c\lam^{j-n_k}.
\end{array}      
\end{equation}
This means that if $\eps_0=\min\left\{\del',c\right\}$, then $\left\{\exp Xy_i : X\in S_k\right\}$, is an $(\eps_0,n_k)-separating$ set for the action of $a$ on $K_C$. We conclude that 
$$h_{top}(K_C,a)\ge \lim_k\frac{1}{n_k}\log|S_k|\ge\lim_k\frac{-\rho\log\eps_k}{n_k}\ge\rho\log\lam>0.$$
Thus we achieve the desired conclusion.
\end{proof}

\end{document}